\RequirePackage{fix-cm}
\documentclass[pdftex,a4paper,10pt]{article} % for arXiv

\usepackage{graphicx}
\usepackage{float} %
\usepackage{tikz}
\usetikzlibrary{calc}
\usetikzlibrary{shapes}
\def\dfrac#1#2{\displaystyle{\begingroup\;#1\;\endgroup\over\;#2\;}} %
\def\fracf#1{\{ #1 \}} %
\def\intervalobig#1#2{\bigl( #1, #2 \bigr)} %
\def\intervalco#1#2{[#1, #2)} %

\usepackage{amsmath, amsthm, amssymb}
\usepackage{url}

\usepackage{physics}
\usepackage{graphicx,lipsum}
\graphicspath{ {./downloads/} }

\usepackage[colorlinks,citecolor=red,urlcolor=blue,bookmarks=false,hypertexnames=true]{hyperref}
\usepackage{tikz}
\usetikzlibrary{calc}
\usetikzlibrary{shapes}

\usepackage[autostyle]{csquotes}
\makeatletter

\usepackage[colorlinks]{hyperref}
\PassOptionsToPackage{dvipsnames}{xcolor}
\usepackage{tikz}
 \definecolor{c350}{rgb}{0.924,0.090,0.000}
 \definecolor{c450}{rgb}{0.472,0.090,0.000}
 \definecolor{c560}{rgb}{0.090,0.842,0.000}

 \definecolor{c120}{rgb}{0.99,0.5,0.000}
 \definecolor{c123}{rgb}{0.618,0.236,0.924}
 \definecolor{c134}{rgb}{0.618,0.924,0.472}
 \definecolor{c145}{rgb}{0.786,0.472,0.090}
 \definecolor{c156}{rgb}{0.618,0.090,0.842}
 \definecolor{c167}{rgb}{0.618,0.842,0.326}
 \definecolor{c178}{rgb}{0.618,0.326,0.972}
 \definecolor{c189}{rgb}{0.618,0.972,0.562}

 \definecolor{c235}{rgb}{0.236,0.924,0.090}

 \definecolor{c050611}{rgb}{0.090,0.708,0.894}

\usepackage[nameinlink,capitalize]{cleveref}
\newtheorem{theorem}{Theorem}%

\newtheoremstyle{named}{}{}{\itshape}{}{\bfseries}{.}{.5em}{\thmnote{#3's }#1}
\theoremstyle{named}

\newtheorem{remark}{Remark}
\theoremstyle{definition}

\theoremstyle{remark}

\theoremstyle{conclusion}

\theoremstyle{observation}

\usepackage{xspace}
\usepackage[margin=1.0in]{geometry}

\usepackage{titlesec}

\usepackage{mathtools}

\DeclarePairedDelimiterX{\inp}[2]{\langle}{\rangle}{#1, #2}
\titleformat{\chapter}
  {\Large\bfseries} 
  {}                
  {0pt}            
  {\huge}

\usepackage{float} %

\def\dfrac#1#2{\displaystyle{\begingroup\;#1\;\endgroup\over\;#2\;}} %
\def\fracf#1{\{ #1 \}} %
\def\intervalobig#1#2{\bigl( #1, #2 \bigr)} %
\def\intervalco#1#2{[#1, #2)} %

\def\comment#1{{}}

\begin{document}
\title{A concise geometric proof of the three distance theorem%
}

\author{Tadahisa Hamada%
}

\date{}

\maketitle

\begin{abstract}
The \emph{three distance theorem} states that for any given irrational number $\alpha$ and a natural number $n$, when the interval $( 0, 1 )$ is divided into $n+1$ subintervals by integer multiples of $\alpha$, namely, $\{0\}, \{ \alpha \}, \{ 2\,\alpha \}$, $\dots$, $\{ n\,\alpha \}$, then each subinterval is limited to at most three different lengths. Steinhaus conjectured this theorem in the 1950s, and many researchers have given various proofs since then.

This paper aims to improve the perspective by showing a two-dimensional map which tells how the unit interval is divided by continuously changing $\alpha$, and provide a concise proof of the theorem.
By illustrating this proof through geometric visualizations, we offer a clearer and more intuitive understanding of the underlying principles and relationships. The approach not only reinforces the classical results but also paves the way for new insights and applications in the study of irrational numbers and their properties.
Additionally, we present a simple proof of the three gap theorem, which is a dual of the three distance theorem.
\end{abstract}
\begin{description}
\item[Keywords]{
three distance theorem, three gap theorem, Farey sequence, Farey pair}
\item[AMS Subj. Class]{11J70, 11J71}
\end{description}

\section{Introduction}
\label{intro}

The historical accounts suggest that Pythagorean school discovered that the sounds produced by strings resonate harmoniously when the ratio of string lengths is $3 : 2$ %
 (for example,
\cite{iamblichus1965vita}, \cite{burkert1972lore}, %
\cite{vanderWaerden1943-VANDHD-2}).
Building on this, we can observe the fact that the three distance theorem can be naturally derived by further exploring the findings from the proportional relationships behind harmonic intervals as follows.

Starting with the note $F$ and successively multiplying the string length by $2/3$  (equivalent to multiplying frequency by $3/2$), and doubling it when less than half the original length, we generate the Pythagorean temperament: $F \to C \to G \to D \to A \to E \to B$. Figure \ref{Pythagorean temperament} illustrates how this temperament subdivides an octave in frequency. 
By multiplying $\log_2 3$ by rank numbers $1, 2, 3, \dots$, and marking the fractional parts (since 1 and 2 correspond to 0 and 1 on a logarithmic scale with base 2), we observe the formation of the Pythagorean temperament.

\begin{figure}[h]
\centering
\begin {tikzpicture}[x=12cm,y=10.3cm]
 \foreach \n in {0,...,5}{
 \draw (0,-\n/7)--(1,-\n/7); \draw (0,-\n/7-0.01)--(0,-\n/7+0.01); \draw (1,-\n/7-0.01)--(1,-\n/7+0.01);
 \draw (-0.01, -\n/7-0.025) node[red]{0}; \draw (1.01, -\n/7-0.025) node[red]{1};
 \draw (0.585,-\n/7-0.01)--(0.585,-\n/7+0.01); \draw (0.585, -\n/7+0.03) node{1};
}
 \foreach \n in {0}{
 \draw[dashed](0,-\n/7-0.01)--(0.585,-\n/7-0.01); \draw (0.585/2, -\n/7-0.03) node{0.585}; 
 \draw (0, -\n/7+0.06) node{2 notes};
}

 \foreach \n in {0,1}{
   \draw(1-0.415/2, -\n/7-0.03)node{0.415};
}

 \foreach \n in {1,...,5}{
 \draw (0.170,-\n/7-0.01)--(0.170,-\n/7+0.01); \draw (0.170, -\n/7+0.03) node{2};
}
 \foreach \n in {1}{
   \draw[dashed](0.170,-\n/7-0.01)--(1,-\n/7-0.01);
 \draw (0, -\n/7+0.06) node{3 notes};
}
 \foreach \n in {1,...,3}{
 \draw (0.170/2, -\n/7-0.03) node{0.170}; 
}

 \foreach \n in {2,...,5}{
 \draw (1-0.245,-\n/7-0.01)--(1-0.245,-\n/7+0.01); \draw (1-0.245, -\n/7+0.03) node{3};
 \draw (0.170*2,-\n/7-0.01)--(0.170*2,-\n/7+0.01); \draw (0.170*2, -\n/7+0.03) node{4};
}
 \foreach \n in {2}{
   \draw[dashed](0.170*2,-\n/7-0.01)--(0.585,-\n/7-0.01);
   \draw[dashed](1-0.245,-\n/7-0.01)--(1,-\n/7-0.01);
   \draw (1-0.245/2, -\n/7-0.03) node{0.245}; 
 \draw (0, -\n/7+0.06) node{5 notes};
}

 \foreach \n in {3,...,5}{
 \draw (1-0.075,-\n/7-0.01)--(1-0.075,-\n/7+0.01); \draw (1-0.075, -\n/7+0.03) node{5};
 \draw (0.170*3,-\n/7-0.01)--(0.170*3,-\n/7+0.01); \draw (0.170*3, -\n/7+0.03) node{6};
}
 \foreach \n in {3}{
   \draw[dashed](0,-\n/7-0.01)--(0.170*3,-\n/7-0.01);
   \draw[dashed](0.585,-\n/7-0.01)--(1-0.075,-\n/7-0.01);
 \draw (0, -\n/7+0.06) node{7 notes};
}
 \foreach \n in {3,4,5}{
   \draw (1-0.075/2-0.01, -\n/7-0.03) node{0.075}; 
}

 \foreach \n in {4,5}{
 \draw (0.095,-\n/7-0.01)--(0.095,-\n/7+0.01); \draw (0.095, -\n/7+0.03) node{7};
 \draw (0.680,-\n/7-0.01)--(0.680,-\n/7+0.01); \draw (0.680, -\n/7+0.03) node{8};
 \draw (0.265,-\n/7-0.01)--(0.265,-\n/7+0.01); \draw (0.265, -\n/7+0.03) node{9};
 \draw (0.850,-\n/7-0.01)--(0.850,-\n/7+0.01); \draw (0.850, -\n/7+0.03) node{10};
 \draw (0.435,-\n/7-0.01)--(0.435,-\n/7+0.01); \draw (0.435, -\n/7+0.03) node{11};
}
 \foreach \n in {4}{
   \draw[dashed](0,-\n/7-0.01)--(0.095,-\n/7-0.01);
   \draw[dashed](0.170,-\n/7-0.01)--(0.265,-\n/7-0.01);
   \draw[dashed](0.340,-\n/7-0.01)--(0.435,-\n/7-0.01);
   \draw[dashed](0.585,-\n/7-0.01)--(0.680,-\n/7-0.01);
   \draw[dashed](0.755,-\n/7-0.01)--(0.850,-\n/7-0.01);
 \draw (0.095/2, -\n/7-0.03) node{0.095}; 
 \draw (0, -\n/7+0.06) node{12 notes};
}

 \foreach \n in {5}{
 \draw (0.020,-\n/7-0.01)--(0.020,-\n/7+0.01); \draw (0.020, -\n/7+0.03) node{12};
 \draw (0.605,-\n/7-0.01)--(0.605,-\n/7+0.01); \draw (0.605, -\n/7+0.05) node{13};
 \draw (0.189,-\n/7-0.01)--(0.189,-\n/7+0.01); \draw (0.189, -\n/7+0.05) node{14};
 \draw (0.774,-\n/7-0.01)--(0.774,-\n/7+0.01); \draw (0.774, -\n/7+0.05) node{15};
 \draw (0.359,-\n/7-0.01)--(0.359,-\n/7+0.01); \draw (0.359, -\n/7+0.05) node{16};
 \draw[dashed](0.020,-\n/7-0.01)--(0.170,-\n/7-0.01);
 \draw[dashed](0.189,-\n/7-0.01)--(0.340,-\n/7-0.01);
 \draw[dashed](0.359,-\n/7-0.01)--(0.585,-\n/7-0.01);
 \draw[dashed](0.605,-\n/7-0.01)--(0.755,-\n/7-0.01);
 \draw[dashed](0.774,-\n/7-0.01)--(1,-\n/7-0.01);
 \draw (0.020, -\n/7-0.05) node{0.020}; 
 \draw (0, -\n/7+0.06) node{17 notes};
}
\end{tikzpicture}
\caption{Formation of Pythagorean temperament}\label{Pythagorean temperament}
\end{figure}
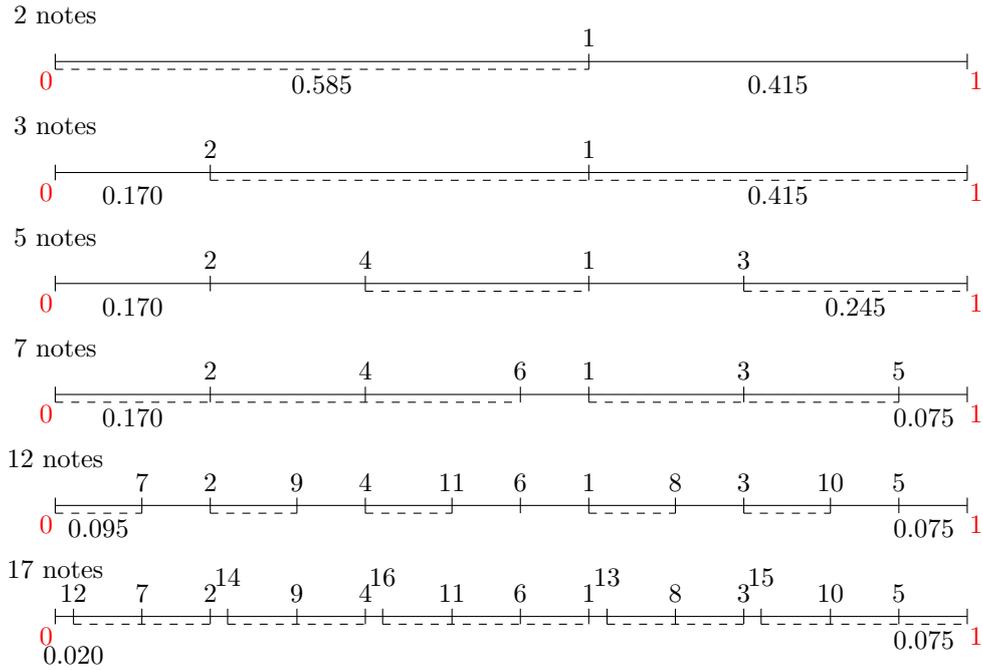

Arranging these values in ascending order and examining the differences between adjacent values, we find that the differences are limited to two or three values. 
Figure \ref{Pythagorean temperament} shows the divisions for two different values of difference, in order of increasing number of divisions. Temperaments with 5, 7, or 12 divisions of an octave are typical in traditional and contemporary music.

We can also see that the number of divisions with two difference values corresponds to the denominators of convergents obtained by expanding $\log_2 3$ into continued fractions and that the interval adjacent to 0 
differs from that adjacent to 1.
Furthermore, the difference in 
rank numbers of adjacent values equals one of the two types of interval numbers modulo $n$, with the sign differing based on whether the difference is a large or small interval. This relationship holds at both ends of an octave if the rank number corresponding to 0 and 1 is set to 0.

These observations lead naturally to the three distance theorem since it is understood that the above holds true even if $\log_2 3$ is replaced by an arbitrary, irrational number, given a continued fraction expansion.
This fact can also be interpreted as an expression of %
regularity of the image of a finite arithmetic progression with the initial term $0$ in the projection from the set of real numbers $\mathbb{R}$ to the unit interval $\mathbb{R} / \mathbb{Z} (= [0, 1))$.
Steinhaus formulated this as a conjecture only in the second half of the 20th century, and many researchers have provided various proofs since then.

Here we give a brief survey of the different approaches adopted by various authors to prove the three distance theorem and its dual, the three gap theorem.

Previous proofs of the three distance theorem have shown that for a given $\alpha$, the lengths of line segments  $\fracf{ i \alpha }$ dividing the unit interval are determined by $\min ( \fracf{ i \alpha } )$ and $\max ( 1 - \fracf{ i \alpha } )$. Throughout the 20th century, proofs were primarily based on elementary number theory concepts like continued fractions and combinatorics, as seen in works like 
\cite{sos1958distribution}, \cite{swierczkowski1958successive}, \cite{halton1965distribution}, \cite{slater1967gaps}, \cite{van1988three}. Alessandri and Berth{\'e} provided a survey paper \cite{alessandri1998three} summarizing these efforts.

In the 21st century, researchers adopted contemporary methodologies. Biringer and Schmidt \cite{biringer2008three} explored the Riemannian manifold, while Haynes, Koivusalo, Walton, and Sadun \cite{haynes2016gaps} examined CW-complexes. Marklof and Str{\"o}mbergsson \cite{marklof2017three} used the modular group with lattice structures, and Taha \cite{taha2017three} utilized properties of Farey fractions. Bockiting-Conrad, Kashina, Petersen, and Tenner \cite{Bockiting-Conrad2021} employed S{\'o}s permutations, and Nagata and Takei \cite{nagata2023} explored a two-dimensional version.

Our contribution lies in providing a new geometric perspective on the three distance theorem. We demonstrate how the unit interval is divided by continuously varying $\alpha$, presenting a two-dimensional map that visualizes this process. This approach not only simplifies the understanding of the theorem but also provides an intuitive visual representation of the relationships involved. Our method unifies the understanding of the theorem for both irrational and rational numbers, thereby offering a comprehensive insight into the behavior of points on the unit interval. 

Traditionally, proofs have discussed $\alpha$ as a fixed specific value. However, our approach visualizes the division of the unit interval for varying $\alpha$, showing how it changes for different $\alpha$ values. This progression enables a shift from understanding discrete sets of values (specific rational and irrational $\alpha$) to a comprehensive understanding over the continuous range (all real numbers $\alpha$ between 0 and 1). 
This method reinforces the classical results and opens new avenues for a comprehensive understanding of the theorem across the entire set of real numbers, including both rational and irrational numbers.
 Additionally, we offer a straightforward proof of the three gap theorem, which serves as a dual to the three distance theorem.

\section{Preparation}
Let $\lfloor x \rfloor\ (x \in \mathbb{R})$ denote the largest integer not exceeding $x$, 
and $\fracf{ x }$ the fractional part of $x \in \mathbb{R}$, i.e., $\fracf{ x } := x - \lfloor x \rfloor$.

The \emph{Farey sequence} \label{fareyn}$F_n$ for any positive integer $n$ consists of irreducible rational numbers 
$\dfrac{a}{b}$ with $0 \le a \le b\le n$ and $\gcd (a, b)=1$, arranged in increasing order.
Consecutive fractions $\bigl( \dfrac{a}{b}, \dfrac{c}{d} \bigr)$ in %
$F_n$ form a \emph{Farey pair}.
The \emph{mediant} of two {fractions $\dfrac{a}{b}$ and $\dfrac{c}{d}$} is defined as $\dfrac{a+c}{b+d}$.
In the Farey sequence $F_{n}$ for $n$,
interpolating the mediants of the Farey pairs where the sum of the denominators is  $n+1$ generates $F_{n+1}$,
the Farey sequence for $n+1$.
 An interval $\bigl( \dfrac{a}{b}, \dfrac{c}{d} \bigr)$ is called a \emph{Farey interval} if $\dfrac{a}{b}$ and $\dfrac{c}{d}$ are a Farey pair in some $F_n$.

Next, we define the \emph{unit interval integral partition} of order $n$ ($U_n$) as a two-dimensional region with $\alpha \in ( 0, 1 )$ on the %
horizontal axis
 and height $h \in \intervalco{0}{1}$ on the %
vertical axis,
 with lines $h = \fracf{ i\,\alpha },\ (0 < \alpha < 1,\ i = 1, \dots , n)$. 
Once  $\alpha$ is determined, the line segment representing %
$h \in \intervalco{0}{1}$ is divided into, at most $n+1$ subintervals.
For convenience, separated individual intervals are assumed to be semi-open intervals, closed to the left and open to the right.
 Figure \ref{n=3} shows the case $n=3$.

\begin{figure}[H]
\centering
\begin {tikzpicture}[x=7cm,y=7cm]
 \draw (1.03,-.02) node{$\alpha$};
 \draw (.99,-.03) node{1};
 \draw (0,-.03) node{0};
 \draw (-.02,1.03) node{$h$};
 \draw (-.03,.99) node{1};
 \draw (-.03,0) node{0};
 \draw (0, 0) rectangle (1,1);
 \draw (0, 0)--(1, 1);
 \draw (0, 0)--(1/2, 1);
 \draw (1/2, 0)--(1, 1);
 \draw [dashed](1/2, 0)--(1/2, 1);
   \draw (1/2, 0) node[below]{$\frac{1}{2}$};
 \draw (0, 0)--(1/3, 1);
 \draw (1/3, 0)--(2/3, 1);
 \draw (2/3, 0)--(1, 1);
 \draw [dashed](1/3, 0)--(1/3, 1);
   \draw (1/3, 0) node[below]{$\frac{1}{3}$};
 \draw [dashed](2/3, 0)--(2/3, 1);
   \draw (2/3, 0) node[below]{$\frac{2}{3}$};
\end{tikzpicture}
\caption{Unit interval integral partition of order 3 ($U_3$):
$h = \fracf{ i \alpha },\ (0 < \alpha < 1,\ i = 1, 2, 3)$}\label{n=3}
\end{figure}
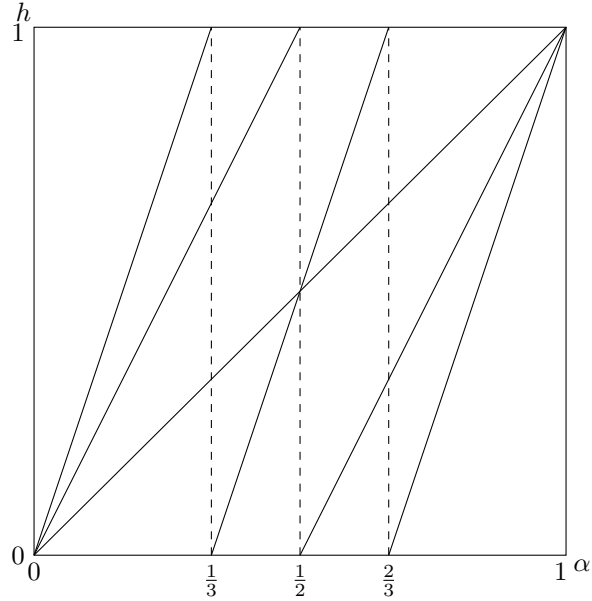

Similar ideas to our unit interval integral partition were introduced by Berstel and Pocchiola \cite{BerstelPocchiola1993} to investigate the number and characteristics of balanced words of length $n$. Subsequently, Yasutomi \cite {yasutomi1998}, Kamae and Takahashi \cite{kamaetakahashi2006} and Akiyama \cite{akiyama2021} discussed in geometric mapping of balanced words.
The unit interval to be divided (vertical axis) in our unit interval integral partition corresponds to the intercept of the Sturmian word in their mapping, whereas the mapping of each order is identical.

\section{The three distance theorem}

We state the three distance theorem by extending it to the case where $\alpha$ is real.

\begin{theorem}[Three distance theorem]
Let $0 < \alpha < 1$ be a %
real
 number and $n$ a positive integer.
The points $\fracf{ i\,\alpha }$, for $0 \le i \le n$, partition the unit interval into at most $n+1$ subintervals,
the lengths of which take at most three values, one being the sum of the other two.

More precisely,  
if $\alpha \in F_n$, then the subintervals of the unit interval divided by $\fracf{ i\,\alpha }$, for $0 \le i \le n$, all have the same length. 
Otherwise, let $\dfrac{a}{b} < \dfrac{c}{d}$ be the Farey pair surrounding $\alpha$.
Then, the unit interval is divided by the points %
 $\fracf{ i\,\alpha }$, for $0 \le i \le n$, into $n+1$ intervals which satisfy:
\begin{itemize}
\item $n+1-b$ of them have length $b\,\alpha - a$,
\item $n+1-d$ of them have length $c-d\,\alpha$,
\item $b+d-(n+1)$ of them have length $(b-d)\,\alpha +c-a$.%
\end{itemize}
The last case occurs when $b+d>n+1$.
\end{theorem}

\begin{proof}
Each line segment inside $U_n$ can generally be expressed as $h = i\,\alpha - j\ (i = 1,\dots, n,\ 0 \le j < i,\ 0 < h < 1)$.
Thus, if two segments $h = i\,\alpha - j$ and $h = i'\,\alpha - j'$ intersect, their $\alpha$-coordinate is an irreducible fraction with a denominator less than $n$, belonging to $F_{n-1}$.

 If $\alpha \in F_n$, then the subintervals of the unit interval divided by $h=\{ i\,\alpha \}$, for $i = 1,\dots, n$, all have the same length. 
Because, 
if $\alpha = \dfrac{p}{q} \in F_n$, the indeterminate equation $p\,i - q\,j = k\ (0 < p < q \le n,\ k = 1,\dots, q-1)$
has solutions for each \(k\) in the range $0 < i \le n,\ 0 \le j < i$, so the set of intersections with each line segment $h = \{ i\,\alpha \}\ (i = 1,\dots, n,\ 0 < h < 1)$ and the line segment $\alpha = \dfrac{p}{q}\ (0 < h < 1)$ is $\bigl\{ \bigl( \dfrac{p}{q}, \dfrac{k}{q} \bigr) \bigm| k = 1,\dots, q-1 \bigr\}$.

Otherwise, %
consider a rectangular region with the corresponding Farey interval on the horizontal axis and the unit interval on the vertical axis $ \intervalobig{\dfrac{a}{b}}{\dfrac{c}{d}} \times (0,\ 1)$.
From the previous discussion, it is clear that all points $(\dfrac{a}{b}, \dfrac{k}{b})\ (k = 1,\dots, b-1)$ that divide the leftmost unit line segment of this Farey interval 
$\bigl\{ \dfrac{a}{b} \bigr\} \times (0,\ 1) $ into $b$ equal parts
 and all points $(\dfrac{c}{d}, \dfrac{k}{d})\ (k = 1,\dots, d-1)$ that 
 divide the rightmost unit line segment of this Farey interval 
$\bigl\{ \dfrac{c}{d} \bigr\} \times (0,\ 1)$ into $d$ equal parts,
 pass through one of the lines $h = \fracf{ i\,\alpha }\ (i = 1,\dots, n)$, and these lines do not intersect inside this rectangle.

Therefore, %
the rectangular region is divided into the following three types of regions by 
$h=\{ i\,\alpha \}\ (i = 1,\dots, n)$:
\begin{itemize}
\item triangles with a base of length $\dfrac{1}{b}$ at the left end and a vertex at the right end of %
the rectangular region,
\item triangles with a base of length $\dfrac{1}{d}$ at the right end and a vertex at the left end of %
the rectangular region,
\end{itemize}
and if $b+d > n+1$, then,
\begin{itemize}
\item trapezoids with one base of length $\dfrac{1}{b}$ at the left end and the other base of length $\dfrac{1}{d}$ at the right end of %
the rectangular region.
\end{itemize}
Given $\alpha$, we immediately see that the length of the subintervals separating the unit interval
$\{ \alpha \} \times (0,\ 1)$
are constant for each of the above types of regions and that the length of the section through the trapezoid is equal to the sum of the lengths of the sections through the two types of triangles above, and that the specific lengths of each are also as stated in the theorem.

Note that when $\alpha$ is the mediant of the Farey pair that surrounds it, the lengths of the subinterval delimited by $\fracf{ i\,\alpha }\ (0 \le i \le n)$ is one less since the lengths of the subinterval delimiting each triangle of the first and second type is equal.
\end{proof}

\begin{remark}
From the argument with $U_n$ in this proof, it follows that when $\alpha \notin F_n$,

\begin{gather}\label{n_max}
n \ge \max(b, d),\\
n+1 \le b+d.
\end{gather}

Next, when $\fracf{ 1\,\alpha }, \dots, \fracf{ n\,\alpha }$ are sorted in increasing order, we denote that $\fracf{ i\,\alpha }$ is followed by $\fracf{ j\,\alpha }$ as $\mathrm{suc}(i) = j$. The minimum value among them is $\fracf{ b\, \alpha }$, which we denote by $s$, and the maximum value is $\fracf{ d\, \alpha }$, which we denote by $1-t$, then the following holds.

\begin{align}
&n+1-b \mathrm{\ intervals\ of\ length\ } s = b\,\alpha - a, \mathrm{\  suc}(i) = i+b,\mathrm{\ for\ } 0 \le i \le n - b,\\
&n+1-d \mathrm{\ }\mathrm{intervals\ of\ length\ } t = c - d\,\alpha, \mathrm{\  suc}(i) = i-d,\mathrm{\ for\ } d \le i \le n,\\
&\begin{aligned}
b+d-(n+1) \mathrm{\ intervals\ of\ le} \mathrm{ngth\ } s+t = (b-d)\,\alpha +(c- a), \\
\mathrm{\  suc}(i) = i+b-d,\mathrm{\ for\ } n-b < i <d.
\end{aligned}\label{b+d-(n+1)}
\end{align}
\end{remark}

\section{The three gap theorem}

The three gap theorem, described next, is a dual to the three distance theorem. The decision parameter $n$ in the three distance theorem is replaced by a real number $\beta$.

\begin{theorem}[Three gap theorem]\label{three gap theorem}
Let $\alpha$, $\beta \in \bigl( 0, 1 \bigr)$ with $\alpha$ irrational.
The gaps between successive integers $i$ for which that $\fracf{i\,\alpha} < \beta$ have at most
three distinct values, one being the sum of the other two.

More precisely, if $\beta \ge \max(\alpha, 1-\alpha)$, the possible gaps are 1 and 2:
\begin{itemize}
\item the gap of 1 occurs with frequency $(2\,\beta-1) / \beta$,
\item the gap of 2 occurs with frequency $(1-\beta) / \beta$.
\end{itemize}
If $\beta < \max(\alpha, 1-\alpha)$, 
let $b$ and $d$ be the smallest positive integers satisfying
\begin{equation*}\label{def_a_b}
\fracf{ b\,\alpha } < \beta,\ \fracf{ d\,\alpha } > 1 - \beta.
\end{equation*}
Then, the possible gaps are $b$, $d$ and $b+d$.
Letting \( s = \{b\,\alpha\} \) and \( t = 1 - \{d\,\alpha \} \):
\begin{itemize}
\item the gap $b$ occurs with frequency $(\beta-s) / \beta$,
\item the gap $d$ occurs with frequency $(\beta-t) / \beta$,
\item the gap $b+d$ occurs with frequency $(s+t-\beta) / \beta$.
\end{itemize}
The last case arises when $s+t > \beta$.
\end{theorem}

\begin{proof}

\noindent{\bf Case 1:} 
For \(\beta \ge \max(\alpha, 1-\alpha)\), since both \(\fracf{0\alpha}\) and \(\fracf{1\alpha}\) are less than \(\beta\), 1 is one of the gaps. Given that \(0 < 1-\beta \le \min(\alpha, 1-\alpha)\), there are clearly two possible gaps: 1 and 2.

In detail, if \(\fracf{i\,\alpha} \in [0, \beta - \alpha) \cup [1 - \alpha, \beta)\), the subsequent gap is 1. If \(\fracf{i\,\alpha} \in [\beta - \alpha, 1 - \alpha)\), the next gap is 2. From number theory, we know that the numbers \(\fracf{i\,\alpha} (i \ge 0)\) are uniformly distributed over the interval [0, 1). Hence, the frequency of each gap matches the length of the associated interval.

\noindent{\bf Case 2:} 
If \(\beta < \max(\alpha, 1-\alpha)\), clearly, \(b \neq d\). Hence, we have
\begin{equation}\label{beta_max}
\beta > \max \{ s, t \}.
\end{equation}
It's straightforward to see that \(\beta \le \fracf{(b-d)\,\alpha}\). If \(\fracf{(b-d)\,\alpha} < \beta\), then \(\fracf{(d-b)\,\alpha} > 1-\beta\), leading to a contradiction either when \(b>d\) or \(b<d\). Given that \(0 < s + t < 1\), we have:
\begin{equation}
\beta \le \fracf{(b-d)\alpha} = s+t.
\end{equation}

Thus, for \(\fracf{i\,\alpha} \in [0, \beta)\), the gaps depend on the value of \(\fracf{i\,\alpha}\).

\begin{equation}
\mathrm{If\ }\fracf{i\,\alpha} \in [0, \beta - s)\mathrm{,\ then\ the\ gap\ after\ }i\mathrm{\ is\ }b,
\end{equation}
because \(\fracf{(i+b)\,\alpha } \in [ s, \beta )\). 
If there exists an \( e \in \mathbb{Z} \) such that \(0 < e < b\) and \(\fracf{ (i+e)\,\alpha } \in [ 0, \beta)\),
then considering \(\{ b\,\alpha \} = s, \{ e\,\alpha \} \in [0, \beta)\) or
\(\fracf{(b-e)\,\alpha} \in [0, \beta)\) contradicts the definition of \(b\).

\begin{equation}
\mathrm{If\ }\fracf{i\,\alpha} \in [t, \beta)\mathrm{,\ then\ the\ gap\ after\ }i\mathrm{\ is\ }d,
\end{equation}
because \(\fracf{(i+d)\,\alpha } \in [ 0, \beta-t )\). 
If there exists an \( e \in \mathbb{Z} \) such that \(0 < e < d\) and \(\fracf{ (i+e)\,\alpha } \in [ 0, \beta)\),
then considering \(\{ d\,\alpha \} = 1-t, \{ e\,\alpha \} \in [1-\beta, 1)\) or
\(\fracf{(d-e)\,\alpha} \in [1-\beta, 1)\) contradicts the definition of \(d\).

\begin{equation}\label{beta-s_t}
\mathrm{If\ }\fracf{i\,\alpha} \in [\beta-s, t)\mathrm{,\ then\ the\ gap\ after\ }i\mathrm{\ is\ }b+d,
\end{equation}
Because, first of all, \( \{(i+b+d)\,\alpha \} \in [\beta-t, s) \). 
There does not exist an \( e \in \mathbb{Z} \) such that \( 0 < e < b+d \) and \( \{(i+e)\,\alpha\} \in [0, \beta) \). 
Let \( u = \{i\alpha\} \in [\beta - s, t) \), and first consider the range \( [u, \beta) \). 
For \( 0 < e < b \), \( \{(i+e)\,\alpha\} \notin [u, u+\beta) \). 
Also, since \( \{(i+b)\,\alpha\} = u+s \in [\beta, s+t) \), 
for \( b < e < b+d \), \( \{(i+e)\alpha\} \notin [s+t-\beta, \beta) \). 
For the range \( [0, u) \), it can be shown in the same way by swapping \(b\) and \(d\).

As in Case 1, the frequency of each gap is equal to the length of the corresponding interval.
\end{proof}

\begin{remark}
Expressions (\ref{beta_max}) through (\ref{beta-s_t}) in the three gap theorem are duals of expressions (\ref{n_max}) through (\ref{b+d-(n+1)}) in the three distance theorem, respectively.
These results align with the dual nature of the expressions in the three distance and three gap theorems, confirming their interconnectedness.
\end{remark}

\bibliographystyle{amsplain}

\end{document}